\documentclass[11pt]{amsart}
\usepackage{amsfonts, amscd, wasysym, enumerate, amsmath, amssymb, color, graphicx, xspace, verbatim, mathtools,paralist,stmaryrd}
\usepackage{axodraw}
\usepackage[all]{xy}
\usepackage[hidelinks]{hyperref}
\mathtoolsset{showonlyrefs=true}

\newcommand*{\un}{{\mathbf 1}}

\newtheorem{thm}{Theorem}

\newtheorem{cor}[thm]{Corollary}

\newtheorem{lem}[thm]{Lemma}
\newtheorem{prop}[thm]{Proposition}
\newtheorem{defn}{Definition}

\newtheorem{rmk}[thm]{Remark}

 \def\strich{\scalebox{0.7}{ 
  \begin{picture}(4,18) (31,-31)
    \SetWidth{1.5}
    \SetColor{Black}
    \Line(32,-14)(32,-30)
\end{picture}
}}

 \def\n{\scalebox{0.7}{ 
  \begin{picture}(4,18) (51,-31)
    \SetWidth{1.5}
    \SetColor{Black}
    \Line(48,-14)(48,-30)
    \Line(64,-14)(64,-30)
    \Line(48,-14)(64,-14)
  \end{picture}
}}

 \def\nin{\scalebox{0.7}{ 
  \begin{picture}(34,18) (31,-31)
    \SetWidth{1.5}
    \SetColor{Black}
    \Line(32,-14)(32,-30)
    \Line(48,-22)(48,-30)
    \Line(64,-14)(64,-30)
    \Line(32,-14)(48,-14)
    \Line(48,-14)(64,-14)
  \end{picture}
}}

\begin{document}

\title[Post-Lie algebras and factorization theorems]{Post-Lie algebras and factorization theorems}


\date{\today}

\maketitle
\begin{center}
\author{
Kurusch Ebrahimi-Fard\footnote{\small Department of Mathematical Sciences, Norwegian University of Science and Technology (NTNU), 7491 Trondheim, Norway. On leave from UHA, Mulhouse, France.\\ {\it\small{kurusch.ebrahimi-fard@ntnu.no}}} 

Igor Mencattini\footnote{\small Instituto de Ci\^encias Matem\'aticas e de Computa\c{c}\~ao, Univ.~de S\~ao Paulo (USP), S\~ao Carlos, SP, Brazil.\\ {\it\small{igorre@icmc.usp.br}}}

Hans Munthe-Kaas\footnote{\small Dept.~of Mathematics, University of Bergen, Postbox 7800,
N-5020 Bergen, Norway.\\ {\it\small{hans.munthe-kaas@math.uib.no}}}
}
\end{center}

\vspace{0.7cm}

\date{today}


\begin{abstract}
In this note we further explore the properties of universal enveloping algebras associated to a post-Lie algebra. Emphasizing the role of the Magnus expansion, we analyze the properties of group like-elements belonging to (suitable completions) of those Hopf algebras. Of particular interest is the case of post-Lie algebras defined in terms of solutions of modified classical Yang--Baxter equations. In this setting we will study factorization properties of the aforementioned group-like elements.
\end{abstract}


\medskip

\begin{quote}
\end{quote}


\tableofcontents


\section{Introduction}
\label{sect:Intro}

This work continues the study of the Lie enveloping algebra of a post-Lie algebra described in \cite{EFLMK}. In a nutshell, a post-Lie algebra is a Lie algebra $\mathfrak g=(V,[\cdot,\cdot])$ whose underlying vector space $V$ is endowed with a bilinear operation, called post-Lie product, satisfying certain compatibility conditions with the Lie bracket $[\cdot,\cdot]$. Since the compatibility of the post-Lie product with $[\cdot,\cdot]$ yields a second Lie bracket $\llbracket\cdot,\cdot\rrbracket$ on $V$, to every post-Lie algebra are naturally associated two Hopf algebras, $\mathcal U(\mathfrak g)$ and $\mathcal U(\bar{\mathfrak g})$, i.e., the universal enveloping algebras of $\mathfrak g$ respectively $\bar{\mathfrak g}=(V,\llbracket\cdot,\cdot\rrbracket)$. Even though $\mathcal U(\mathfrak g)$ and $\mathcal U(\bar{\mathfrak g})$ are neither isomorphic as Hopf algebras nor as associative algebras, one can show that a lift of the post-Lie product to $\mathcal U(\mathfrak g)$ yields a new Hopf algebra $\mathcal U_\ast(\mathfrak g)$ which turns out to be isomorphic as Hopf algebra to $\mathcal U(\bar{\mathfrak g})$. The existence of such a Hopf algebra isomorphism can be thought of as a non-commutative extension of a well-known result proven by Guin and Oudom in \cite{OudomGuin} in the context of pre-Lie algebras.

The present work has two central aims. The first one is to explore several of the results in the papers \cite{RSTS,STS3} from the perspective offered by the relatively new theory of post-Lie algebras \cite{LMK1,LMK2,Vallette}. The second aim is to start a more systematic investigation of the so called post-Lie Magnus expansion introduced in \cite{EFLIMK}, both from the point of view of its properties as well as its applications to isospectral flows. More details, see, for example, \cite{ChuNorris,Watkins} and the monograph \cite{Suris}.

As noticed for the first time in \cite{GuoBaiNi}, a rich source of concrete examples of post-Lie algebra is provided by the theory of classical $r$-matrices together with the corresponding classical Yang--Baxter equations, which play an important role in the theory of classical integrable systems \cite{BaBeTa,PolishReview,STS2, Suris}. It is worth noticing that there exist actually two different definitions of classical $r$-matrices which are not completely equivalent; the first one, due to Drinfeld, gives rise to the structure of Lie algebra on the dual space of a given Lie algebra. The second one, due to Semenov-Tian-Shansky, yields a second Lie bracket on the same underlying linear space. Accordingly, one speaks of a \emph{Lie bialgebra} in the former case and of a \emph{double Lie algebra} in the latter one. The role of these definitions is different. Lie bialgebras arise in connection with the deformation of the co-commutative coproduct on the universal enveloping algebra of the initial Lie algebra, and eventually they go together with the construction of the deformed algebra $\mathcal U_q(\mathfrak g)$. Double Lie algebras, on the contrary provide abstract versions of factorization problems which play the central role in the study of classical integrable systems admitting a Lax representation. There exists also a way to combine both definitions yielding the notion of \emph{factorizable Lie bialgebras}, see \cite{RSTS}. It is this latter version that is used to extend factorization theorems from the classical realm to quantum algebras $\mathcal U_q(\mathfrak g)$; the extra condition imposed on the classical $r$-matrix in this case is the skew-symmetry (with respect to the invariant inner product on $\mathfrak g$). There exist many double Lie algebras for which the associated $r$-matrix is not skew; in this case factorization theorems are still valid, as pointed out in \cite{STS3}, but there is, in general, no natural way to deform the coproduct (in the category of Hopf algebras). In the present work we will deal exclusively with the case of factorization theorems for \emph{ordinary} universal enveloping algebra, leaving the case of quantum algebras for future investigations. In particular, we will adopt systematically the notation and terminology used in \cite{STS1}.

As already remarked, in the seminal work  \cite{STS1} Semenov-Tian-Shansky showed that solutions of modified classical Yang--Baxter equations, known as classical $r$-matrices, play an important role in studying solutions of Lax equations, and are intimately related to particular factorization problems in the corresponding Lie groups. More precisely, any solution $R$ of the modified classical Yang--Baxter equation on a Lie algebra $\mathfrak g$ gives rise to a so-called double Lie algebra, i.e., a second Lie algebra $\mathfrak g_R$ can be defined on the vector space underlying $\mathfrak g$. Its Lie bracket is given in terms of the original Lie bracket of $\mathfrak g$ together with the classical $r$-matrix $R$, in such a way, that when splitting the linear map $R = R_+ + R_-$  appropriately, both maps, $R_{\pm}$, become Lie algebra morphisms from $\mathfrak g_R$ to $\mathfrak g$. Every element of the Lie group $G$ corresponding to $\mathfrak g$, which is sufficiently close to the identity, admits a factorization as a product of two elements belonging to two suitably defined Lie subgroups $G_{\pm} \subset G$. See \cite{STS4} for more details. It is this sort of factorization that plays a critical role in the solution of the isospectral flow mentioned above. As an aside, we remark that the latter are closely related to matrix factorization schemes \cite{ChuNorris,Faybusovich,Watkins}.     

In an attempt to extend this analysis to the theory of quantum integrable systems, the aforementioned factorization problem has been studied in references \cite{RSTS,STS3} in the framework of universal enveloping algebra of a Lie algebra endowed with a solution of the modified classical Yang--Baxter equation. In these works it was shown that every classical $r$-matrix $R$ defined on a (finite dimensional) Lie algebra $\mathfrak g$, gives rise to a factorization of any group-like element of (a suitable completion of) the universal enveloping algebra $\mathcal U(\mathfrak g)$. This result came as a consequence of the existence of a linear isomorphism $F:\mathcal U(\mathfrak g_R)\rightarrow \mathcal U(\mathfrak g)$, extending the identity map between the Lie algebras $\mathfrak g_R$ and $\mathfrak g$. The map $F$ is defined explicitly in terms the usual Hopf algebra structures on the corresponding universal enveloping algebras, $\mathcal U(\mathfrak g_R)$ and $\mathcal U(\mathfrak g)$, together with the liftings of the Lie algebra morphisms $R_{\pm}$, defined via the aforementioned splitting of $R$, to unital algebra morphisms between those algebras. In the paper \cite{STS3} a new associative product was defined on $\mathcal U(\mathfrak g)$ by pushing-forward the associative product of $\mathcal U(\mathfrak g_R)$ in terms of the linear isomorphism $F$, making it an isomorphism of unital associative algebras. See also \cite{RSTS}. 

The Hopf algebraic results for general post-Lie algebra motivate our aim to reconnoiter references \cite{RSTS,STS3} from a Hopf algebra theoretic point of view using the post-Lie product induced by a classical $r$-matrix. Indeed, we shall show that when a post-Lie algebra structure is defined in term of a solution of the modified classical Yang--Baxter equation, the aforementioned Hopf algebra isomorphism between $\mathcal U_\ast(\mathfrak g)$ and $\mathcal U(\bar{\mathfrak g})=\mathcal U(\mathfrak g_R)$ can be realized in terms of the Hopf algebra structure of these two universal enveloping algebras. It assumes the explicit form of the map $F$ introduced in \cite{RSTS,STS3}. We deduce that the associative product defined \cite{RSTS,STS3} as the push-forward to $\mathcal U(\mathfrak g)$ of the product of $\mathcal U(\mathfrak g_R)$ coincides with the extension to $\mathcal U(\mathfrak g)$ of the post-Lie product defined on $\mathfrak g$ in terms of the classical $r$-matrix. As a practical consequence this makes the computation of the product originally defined in \cite{RSTS,STS3} more transparent. The aforementioned is based on the central part of this work, which aims at understanding the role of post-Lie algebra in the context of the factorization problem mentioned above. In this respect we show for any post-Lie algebra, that for every $x \in \mathfrak g$ there exist a unique element $\chi(x) \in \mathfrak g$, such that $\exp(x) = \exp^*(\chi(x))$ in (suitable completions of) $\mathcal U(\mathfrak g)$ and $\mathcal U_*(\mathfrak g)$. The map $\chi: \mathfrak g \to \mathfrak g$ is described as the solution of a particular differential equation, and is dubbed post-Lie Magnus expansion. We show that in the classical $r$-matrix case this general post-Lie result implies, that any group-like element $\exp(x)$ in (a suitable completion of) $\mathcal U(\mathfrak g)$ factorizes into the product of two exponentials, $ \exp(\chi_+(x))$ and $\exp(\chi_-(x))$, with $\chi_{\pm}(x):=R_{\pm}\chi(x)$. In forthcoming work we intend to explore in greater detail the -- post-Lie -- algebraic and geometric properties of the map $\chi$ and the corresponding factorization from the point of view of Riemann--Hilbert problems related to the study of solutions of Lax equations \cite{RSTS,STS1}.  

\smallskip

\emph{The paper is organized as follows}. After recollecting the definition of a post-Lie algebra and some of its most elementary properties in Section \ref{sect:LieAdmPostLie}, we recall, for the sake of completeness, some basic information about the theory of the classical $r$-matrices. Then, with the aim of being as self-contained as possible, we discuss how the lifting of the post-Lie product yields the Hopf algebra $\mathcal U_*(\mathfrak g)$, which is isomorphic as a Hopf algebra to $\mathcal U(\bar{\mathfrak g})$. The new result in this section is Theorem \ref{thm:FinverseChi}. In Section \ref{sect:anotherHA} yet another seemingly different looking Hopf algebra on $\mathcal{U}(\mathfrak g)$ is introduced in the specific context of a Lie algebra $\mathfrak g$ endowed with a classical $r$-matrix. This Hopf algebra is then shown to coincide with the one coming from the post-Lie algebra induced on $\mathfrak g$ by the classical $r$-matrix. Finally in Section \ref{sect:factorThm} we explore a natural factorization theorem for group-like elements using Theorem \ref{thm:FinverseChi} in the appropriately completed universal enveloping algebra.

\smallskip

\begin{rmk}
In this work all vector spaces are assumed to be finite dimensional over the base fields $\mathbb K=\mathbb R$ or $\mathbb K=\mathbb C$.  Moreover, often we will need to consider different Lie algebra structures defined on the same underlying vector space, which, from now on, will be denoted with $V$. 
\end{rmk}

\smallskip

\noindent {\bf{Acknowledgements}}: This work started during a stay of the first author at the Instituto de Ci\^encias Matem\'aticas e de Computa\c{c}\~ao, Univ.~de S\~ao Paulo, campus S\~ao Carlos, Brazil, which was support by the FAPESP grant 2015/06858-2.


\section{Post-Lie algebras and classical $r$-matrices}
\label{sect:LieAdmPostLie}

We start this section by recalling the definition of post-Lie algebra \cite{MKW08,LMK1,Vallette} together with some of its basic properties. See \cite{EFLMK} for more details and references. We will also briefly discuss the post-Lie algebra structure on a Lie algebra endowed with a solution of modified classical Yang--Baxter equation (MCYBE). See \cite{GuoBaiNi} for more details. Then we summarize how post-Lie algebra properties are lifted to the universal enveloping algebra of the corresponding Lie algebra. Details can be found in \cite{EFLIMK}.

Let $(\mathcal A,\cdot)$ be a $\mathbb K$-algebra. Recall the definition of the {\it{associator}} map ${\rm{a}}_\cdot : \mathcal A \otimes \mathcal A \otimes \mathcal A \to \mathcal A$   
$$
	{\rm{a}}_\cdot(x,y,z):= x \cdot (y \cdot z) - (x \cdot y) \cdot z,
$$
for any $x,y,z\in \mathcal A$. The definition of post-Lie algebra follows.

\begin{defn} \label{def:postLie} 
Let $\mathfrak g=(V, [\cdot,\cdot])$ be a Lie algebra, and let $\triangleright  : V \otimes V \rightarrow V$ be a binary product such that for all $x,y,z \in V$
\begin{equation}
\label{postLie1}
	x \triangleright [y,z] = [x\triangleright y , z] + [y , x \triangleright z],
\end{equation}
and
\begin{equation}
\label{postLie2}
	[x,y] \triangleright z = {\rm{a}}_{\triangleright  }(x,y,z) - {\rm{a}}_{\triangleright  }(y,x,z).
\end{equation}
Then the pair $(\mathfrak{g}, \triangleright)$ is called a \emph{left post-Lie algebra}. 
\end{defn}

\begin{rmk}
\begin{enumerate}[i)]
\item From now on, given a post-Lie algebra $(\mathfrak g,\triangleright)$, we will write $x\in\mathfrak g$ instead of $x\in V$.
\item Relation \eqref{postLie1} implies that for every left post-Lie algebra the natural map $\ell_\triangleright : \mathfrak g \rightarrow\operatorname{End}_{\mathbb K}(\mathfrak g)$ defined by $\ell_\triangleright (x)(y) := x \triangleright y$ is linear and takes values in the derivations of $\mathfrak g$.

\item Together with the notion of left post-Lie algebra one can introduce the notion of \emph{right post-Lie algebra}. In this case \eqref{postLie2} becomes $[x,y] \triangleright z = {\rm{a}}_{\triangleright  }(y,x,z) - {\rm{a}}_{\triangleright  }(x,y,z).$
\end{enumerate}
\end{rmk}

For the rest of this work, unless stated otherwise, the term post-Lie algebra refers to left post-Lie algebra. Furthermore, the next result is critical to the theory of post-Lie algebras. 

\begin{prop} \cite{LMK1} \label{prop:post-lie}
Let $(\mathfrak g, \triangleright)$ be a post-Lie algebra. The bracket
\begin{equation}
\label{postLie3}
	[[x,y]] := x \triangleright y - y \triangleright x + [x,y]
\end{equation}
satisfies the Jacobi identity for all $x, y \in \mathfrak g$. 
\end{prop}

Recall that a $\mathbb K$-algebra $(\mathcal A,\cdot)$ is called \emph{Lie admissible} if the commutator $[\cdot,\cdot]: \mathcal A \otimes \mathcal A \rightarrow \mathcal A$, which is defined for all $x,y \in \mathcal A$ by antisymmetrization, $[x,y]:=x \cdot y - y \cdot x$, yields a Lie bracket.  \emph{Left pre-Lie algebras} \cite{Burde,Cartier11,Manchon}, which are characterised through a binary product $\curvearrowleft: \mathcal A \otimes \mathcal A \rightarrow \mathcal A$  satisfying the \emph{left pre-Lie relation} ${\rm{a}}_{\curvearrowleft}(x,y,z) = {\rm{a}}_{\curvearrowleft}(y,x,z),$ are Lie admissible. Likewise a right pre-Lie algebra is defined by ${\rm{a}}_{\curvearrowright}(x,y,z) = {\rm{a}}_{\curvearrowright}(x,z,y)$.
In particular note that, although a post-Lie algebra is not Lie-admissible, one can define the product $x \succ y := x \triangleright y + \frac{1}{2}[x,y],$ such that $(\mathfrak g,\succ)$ is Lie-admissible. 
Moreover, if $(\mathfrak g,\triangleright)$ is a post-Lie algebra, whose underlying Lie algebra $\mathfrak g=(V,[\cdot,\cdot])$ is abelian, i.e. if $[\cdot,\cdot]$ is identically zero, axiom {\rm{(\ref{postLie2})}} reduces to the left pre-Lie identity $\mathrm{a}_{\triangleright}(x,y,z) = \mathrm{a}_{\triangleright}(y,x,z).$ This implies that the vector space $V$ together with the product $\triangleright: V\otimes V \rightarrow V$ is a {\it{left pre-Lie}} algebra.

\begin{rmk}\label{rem:notrem} A few remarks are in order.
\begin{enumerate}
	\item From now on, for given a post-Lie algebra $(\mathfrak g,\triangleright)$, where $\mathfrak g=(V,[\cdot,\cdot])$, we write  $\overline{\mathfrak g}:=(V,\llbracket\cdot,\cdot\rrbracket)$, where $\llbracket\cdot,\cdot\rrbracket$ is the Lie bracket defined in \eqref{postLie3}.	
	
	\item If one trades right for left post-Lie algebras, then the new bracket in Proposition \ref{prop:post-lie}, which satisfies the Jacobi identity, becomes
\[
	[[x,y]]:= x \triangleright y - y \triangleright x -[x,y],\qquad \forall x,y\in\mathfrak g.
\]
	
	\item It turns out that differential geometry is a natural place to look for examples of pre- and post-Lie algebras. Indeed, regarding the former, the canonical connection on $\mathbb{R}^n$ is flat with zero torsion, and defines a pre-Lie
algebra on the set of vector fields. Following \cite{LMK1} a Koszul connection $\nabla$ yields a $\mathbb{R}$-bilinear product $X\triangleright Y=\nabla_XY$ on the space of smooth vector fields $\mathcal{X}(\mathcal{M})$ on a manifold $\mathcal{M}$. Flatness and constant torsion, together with the Bianchi identities imply relation \eqref{postLie3} between the Jacobi-Lie bracket of vector fields, the torsion itself, and the product defined in terms of the connection. 
	
	\item Post-Lie algebras are important in the theory of numerical methods for differential equations. We refer the reader to \cite{EFLMK,MKW08,LMK1} for more details on this topic.
\end{enumerate}
\end{rmk}


{\bf{Classical $r$-matrices.}}\quad We briefly recall a few facts about classical $r$-matrices. For details and examples the reader is referred to \cite{Poisson1, Poisson2, STS2, Suris}. Let $\mathfrak g=(V,[\cdot,\cdot])$ be a Lie algebra and let $\theta \in \mathbb K$ be a parameter fixed once and for all. For a linear map $R$ on $\mathfrak{g}$ the bracket
\begin{equation} \label{eq:Rbra}
	[x,y]_R := \frac{1}{2}([Rx,y]+[x,Ry])
\end{equation}
is skew-symmetric for all $x,y\in\mathfrak{g}$. Moreover, if $B_R:\mathfrak g\otimes\mathfrak g\rightarrow\mathfrak g$ is defined  for all $x,y\in\mathfrak g$ by
\begin{equation}\label{eq:B}
	B_R(x,y):=R([Rx,y]+[x,Ry])-[Rx,Ry],
\end{equation}
then $[\cdot,\cdot]_R$ satisfies the Jacobi identity if and only if:
\begin{equation}\label{eq:modYB}
	[B_R(x,y),z]+[B_R(z,x),y]+[B_R(y,z),x]=0,
\end{equation}
for all $x,y,z\in\mathfrak g$. Defining $B_R(x,y):=\theta [x,y]$, which amounts to the identity
\begin{equation}\label{eq:EMYB}
	[Rx,Ry]=R([Rx,y]+[x,Ry])-\theta [x,y],
\end{equation}
for all $x,y\in\mathfrak{g}$, implies that \eqref{eq:modYB} is fulfilled.

\begin{defn}[Classical $r$-matrix and MCYBE]
\label{def:r-matrix}
Equation \eqref{eq:EMYB} is called {\it modified classical Yang--Baxter Equation} (MCYBE) with parameter $\theta$ and its solutions are called {\it classical $r$-matrices}. For $\theta=0$, equation \eqref{eq:EMYB} reduces to the so called {\it classical Yang--Baxter Equation} (CYBE). 
\end{defn}

In the present work we will be mainly concerned with the case where $\theta=1$. For this reason, in what follows, the term classical $r$-matrix refers to an element $R \in\operatorname{End}_{\mathbb K}(\mathfrak g)$  such that
\begin{equation}
\label{eq:impc}
	[Rx,Ry]=R([Rx,y]+[x,Ry])-[x,y],
\end{equation}
for any $x,y \in \mathfrak g$. In this setting equation \eqref{eq:impc} will be referred to as MCYBE. 

\noindent We call \eqref{eq:Rbra} the {\it{double Lie bracket}} and denote the corresponding Lie algebra by $\mathfrak g_R:=(V,[\cdot,\cdot]_R)$. The Lie algebra $\mathfrak g$ with classical $r$-matrices $R$ is called {\it{double Lie algebra}}. The significance of solutions of \eqref{eq:impc} stems from the following well known result.

\begin{prop}\cite{STS1} One can prove that:
\begin{enumerate}
	\item[i)] The linear maps $R_\pm:\mathfrak g_R\rightarrow\mathfrak g$, defined by 
		\begin{equation}
		\label{eq:rpm}
			R_\pm:=\frac{1}{2}(R\pm\operatorname{id})
		\end{equation}
are Lie algebra morphisms, $R_{\pm}([x,y]_R)=[R_{\pm}x,R_{\pm}y]$, which amounts to the two identities
\begin{equation}
\label{eq:RBE}
	[R_\pm x, R_\pm y] = R_\pm ([R_\pm x, y] + [x,R_\pm y] \mp [x,y]).
\end{equation}

\item[ii)] Moreover, if one defines $\mathfrak g_\pm:=\operatorname{Im}(R_\pm:\mathfrak g_R\rightarrow\mathfrak g)$ and $\kappa_\pm:=\operatorname{Ker}(R_\mp:\mathfrak g_R\rightarrow\mathfrak g)$, then $\kappa_\pm\subset\mathfrak g_\pm$ and the natural application $\Theta:\mathfrak g_+/\kappa_+\rightarrow\mathfrak g_-/\kappa_-$ is an isomorphism of Lie algebras.
\end{enumerate}
\end{prop}

\noindent Observe that the Lie bracket \eqref{eq:Rbra} expressed in terms of the maps $R_\pm$ defined in \eqref{eq:rpm} becomes 
$$
	[x,y]_R = [R_\pm x, y] + [x,R_\pm y] \mp [x,y].
$$
Assume $R \in \operatorname{End}_{\mathbb K}(\mathfrak g)$ to be a solution of equation \eqref{eq:impc}. Let $G$ and $G_R$ be the connected and simply-connected Lie groups corresponding to the Lie algebras $\mathfrak g$ respectively $\mathfrak g_R$. By $r_\pm: G_R \rightarrow G$ we denote the Lie group homomorphisms, which integrate the Lie algebra homomorphisms $R_\pm$. Furthermore, let $G_\pm:=\operatorname{Im}\,(r_\pm:G_R\rightarrow G)$, and let $\delta: G\rightarrow G\times G$ and $i:G\rightarrow G$ be the \emph{diagonal map} and the \emph{inversion map}, respectively, that is, $\delta(g):=(g, g)$ and $i(h):=h^{-1}$. Denoting by $m$ the multiplication of $G$, we define $\tilde m: G \times G \rightarrow G$ to be the map $m \circ (\operatorname{id}_G , i)$, i.e., the map such that $\tilde m(g,h)=m(g,h^{-1})=gh^{-1}$, for all $(g,h) \in G \times G$. Then one can prove the following theorem \cite{STS1}. See also \cite{Faybusovich}.

\begin{thm} \cite{STS1}\label{thm:factorizationtheorem}
The map $I_R:G_R\rightarrow G\times G$, defined for all $g \in G_R$ by 
\[
	I_R (g)=(r_+,r_-)\circ\delta(g)=(r_+g,r_-g),
\]
is an \emph{embedding} of Lie groups. Moreover, the map $\tilde m\circ I_R: G_R \rightarrow G$, defined for all $g \in G_R$ by 
\[
\tilde m\circ I_R(g)=m(r_+g,({r_-g})^{-1})=r_+g({r_-g})^{-1},
\] 
is a local diffeomorphism from a suitable neighborhood of the identity $e\in G_R$ to a suitable neighborhood of the identity $e\in G$. In other words, any element $g \in G$ \emph{sufficiently close to the identity} admits a unique factorization as
\begin{equation}
\label{eq:factlie}
	g=g_+({g_-})^{-1},
\end{equation}
where $(g_+,g_-)\in\text{Im}\,I_R$.
\end{thm} 

\begin{rmk}  \cite{STS1}
The Lie bracket \eqref{eq:Rbra} defined by a  classical $r$-matrix on $\mathfrak g$ implies a corresponding linear Poisson structure $\{\cdot,\cdot\}_R$ on the dual $\mathfrak g^*$. It associates to each Casimir function with respect to the Poisson bracket $\{\cdot,\cdot\}_{\mathfrak g}$ a non-trivial first integral of the original dynamical system.
\end{rmk}

In the following we consider post-Lie algebras defined in terms of $r$-matrices. Let $\mathfrak g$ be a Lie algebra with $R \in \operatorname{End}_{\mathbb{K}}(\mathfrak g)$ a solution of \eqref{eq:impc}, and let $R_-$ be defined as in \eqref{eq:rpm}. 

\begin{thm}\cite{GuoBaiNi}
\label{thm:postLie1}
For any elements $x,y \in \mathfrak g$ the bilinear product 
\begin{equation}
\label{def:RBpostLie}
	x \triangleright y := [R_-x,y]
\end{equation} 
defines a post-Lie algebra structure on the Lie algebra $\mathfrak g$.
\end{thm} 

\begin{proof}
Axiom \eqref{postLie1} holds true since for all $x \in \mathfrak g$, the map $[R_-x,\cdot ]=\mathrm{ad}_{R_-x}: \mathfrak g \to \mathfrak g$ is a derivation with respect to the Lie bracket $[\cdot,\cdot]$. Axiom \eqref{postLie2} follows from identity \eqref{eq:RBE} together with the Jacobi identity.
\end{proof}

\begin{rmk} \label{rmk:doubleLiebracket-post-Lie}
The product $x\triangleleft y:=[R_+x,y]$, defined for all $x,y\in\mathfrak g$, yields on $\mathfrak g$ the structure of a right post-Lie algebra. In particular, note that 
\[
	x\triangleleft y	=[R_+x,y]=\frac{1}{2}[Rx,y]+\frac{1}{2}[x,y]
\]
which implies that $x\triangleleft y=x\triangleright y+[x,y]$, for all $x,y \in\mathfrak g$. Moreover, a simple computation shows that: $x\triangleright y-y\triangleright x+[x,y]=[x,y]_R=x\triangleleft y-y\triangleleft x-[x,y]$, for all $x,y\in\mathfrak g$, which implies that the Lie bracket in \eqref{postLie3} coincides with the Lie bracket in \eqref{eq:Rbra}, i.e., $[[\cdot,\cdot]]=[\cdot,\cdot]_R$.
\end{rmk}


{\bf{The universal enveloping algebra of a Post-Lie algebra.}}\quad In Proposition \ref{prop:post-lie} it is shown that any post-Lie algebra is endowed with two Lie brackets, $[\cdot,\cdot]$ and $[[\cdot,\cdot]]$, which are related in terms of the post-Lie product by identity \eqref{postLie3}. The relation between the corresponding universal enveloping algebras was explored in \cite{EFLMK}. In \cite{OudomGuin} similar results in the context of pre-Lie algebras and the symmetric algebra $\mathcal{S}(\mathfrak g)$ appeared. 

Recall that the universal enveloping algebra $\mathcal U(\mathfrak g)$ of a Lie algebra $\mathfrak g$ is a connected, filtered, noncommutative, cocommutative Hopf algebra with unit $\un$ \cite{Kassel,Sweedler}. Elements in $\mathcal U(\mathfrak g)$ are denoted as words $x_1 \cdots x_n$, and the letters $x_i \in \mathfrak g \hookrightarrow \mathcal U(\mathfrak g)$ are primitive, that is, the coproduct $\Delta x_i = x_i \otimes {\bf 1} + {\bf 1} \otimes x_i$. Its multiplicative extension defines the -- unshuffle -- coproduct on all of $\mathcal U(\mathfrak g)$. The counit $\epsilon : \mathcal U(\mathfrak g) \to \mathbb{K}$ and antipode $S: \mathcal U(\mathfrak g) \to \mathcal U(\mathfrak g) $ are defined by $\epsilon({\bf{1}})=1$ and zero else, respectively $S(x_1\cdots x_n):=(-1)^n x_n\cdots x_1$. In the following Sweedler's notation is used to denote the coproduct $\Delta A=A_{(1)} \otimes A_{(2)}$ for any $A$ in $\mathcal U(\mathfrak g)$. 

The next proposition summarises the results relevant for the present discussion of lifting the post-Lie algebra structure to $\mathcal U(\mathfrak g)$. In what follows we will denote with  $\triangleright$ both the original post-Lie product on $\mathfrak g$ and the one lifted to $\mathcal U(\mathfrak g)$.
\begin{prop}\cite{EFLMK}\label{prop:post1}
Let $A,B,C\in\mathcal U(\mathfrak g)$ and $x,y\in\mathfrak g \hookrightarrow \mathcal U(\mathfrak g),$ then there exists a unique extension of the post-Lie product from $\mathfrak g$ to $\mathcal U(\mathfrak g)$, given by:
\allowdisplaybreaks{
\begin{align}
	{\bf 1}\triangleright A &= A, \quad\ A \triangleright{\bf 1} = \epsilon (A){\bf 1}   \label{eq:pha1}\\
	\epsilon(A\triangleright B) &=\epsilon(A)\epsilon (B),\\
	\Delta (A\triangleright B)&=(A_{(1)}\triangleright B_{(1)}) \otimes (A_{(2)}\triangleright B_{(2)}),\\
	xA\triangleright B&=x\triangleright (A\triangleright B)-(x\triangleright A)\triangleright B\nonumber\\
	A\triangleright BC &=(A_{(1)}\triangleright B)(A_{(2)}\triangleright C).	\label{eq:pha2}
\end{align}}
\end{prop}

\begin{proof}
The proof of Proposition \ref{prop:post1} goes by induction on the length of monomials in $\mathcal U(\mathfrak g)$. 
\end{proof}

Note that \eqref{eq:pha1} together with \eqref{eq:pha2} imply that the extension of the post-Lie product from $\mathfrak g$ to $\mathcal U(\mathfrak g)$  yields a linear map $d:\mathfrak g \rightarrow \operatorname{Der}\big(\mathcal U(\mathfrak g)\big),$ defined via $d(x)(x_1\cdots x_n):=\sum_{i=1}^nx_1\cdots (x\triangleright x_i)\cdots x_n$, for any word $x_1\cdots x_n \in \mathcal U(\mathfrak g)$. A simple computation shows that, in general, this map is not a morphism of Lie algebras. Together with Proposition \ref{prop:post1} one can prove the next statement.

\begin{prop}\cite{EFLMK}\label{prop:post2}
Let $A,B,C\in\mathcal U(\mathfrak g)$ 
\allowdisplaybreaks{
 \begin{align}
	A\triangleright (B\triangleright C)&=(A_{(1)}(A_{(2)}\triangleright B))\triangleright C. \label{last}
\end{align}}
 \end{prop}
 
It turns out that identity \eqref{last} in Proposition \ref{prop:post2} can be written $A\triangleright (B\triangleright C) = m_\ast(A\otimes B) \triangleright C$, where the product $m_\ast:\mathcal U(\mathfrak g)\otimes\mathcal U(\mathfrak g)\rightarrow\mathcal U(\mathfrak g)$ is defined by
\begin{equation}
\label{eq:postLieU}
	m_\ast(A\otimes B)= A\ast B:=A_{(1)}(A_{(2)}\triangleright B).
\end{equation}

\begin{thm}\cite{EFLMK}\label{thm:KLM0}
The product defined in \eqref{eq:postLieU} is non-commutative, associative and unital. Moreover, $\mathcal U_*(\mathfrak g):=(\mathcal U(\mathfrak g),m_\ast,{\bf 1},\Delta,\epsilon,S_\ast)$ is a co-commutative Hopf algebra, whose unit, co-unit and coproduct coincide with those defining the usual Hopf algebra structure on $\mathcal U(\mathfrak g)$. The antipode $S_\ast$ is given uniquely by the defining equations $
	m_\ast\circ(\operatorname{id}\otimes S_\ast)\circ\Delta
	={\bf 1}\circ\epsilon
	=m_\ast\circ(S_\ast\otimes\operatorname{id})\circ\Delta.$
More precisely
\begin{equation}
	S_\ast (x_1\cdots x_n)=-x_1\cdots x_n-\sum_{k=1}^{n-1}
	\sum_{\sigma\in\Sigma_{k,n-k}}x_{\sigma(1)}\cdots x_{\sigma(k)}\ast 
	S(x_{\sigma(k+1)}\cdots x_{\sigma(n)}),\label{eq:antipodests}
\end{equation}
for every $x_1\cdots x_n\in\mathcal U_n(\mathfrak g)$ and for all $n\geq 1$.
\end{thm}

\noindent Equation \eqref{eq:antipodests} becomes clear by noting that since elements $x \in \mathfrak g$ are primitive and $\Delta$ is an algebra morphism with respect to the product \eqref{eq:postLieU}, one deduces  

\begin{lem}\label{lem:coprodast}
For $x_1\ast\cdots\ast x_n \in \mathcal U_*(\mathfrak g)$
\allowdisplaybreaks{
\begin{eqnarray*}
	\Delta(x_1\ast\cdots\ast x_n)
	&=&x_1\ast\cdots\ast x_n\otimes {\bf{1}} \nonumber +{\bf{1}}\otimes x_1\ast \cdots \ast x_n\\
	&+&\sum_{k=1}^{n-1}\sum_{\sigma\in\Sigma_{k,n-k}}
			x_{\sigma(1)} \ast \cdots \ast x_{\sigma(k)}\otimes x_{\sigma(k+1)}\ast \cdots \ast x_{\sigma(n)}.
\end{eqnarray*}}
Here $\Sigma_{k,n-k} \subset \Sigma_n$ denotes the set of permutations in the symmetric group $ \Sigma_n$  of n elements $[n]:=\{1, 2, \ldots, n\}$ such that ${\sigma(1)}< \cdots <\sigma(k)$ and ${\sigma(k+1)}< \cdots <{\sigma(n)}$.
\end{lem}

The relation between the Hopf algebra $\mathcal U_*(\mathfrak g)$ in Theorem \ref{thm:KLM0} and the universal enveloping algebra $\mathcal U(\overline{\mathfrak g})$ corresponding to the Lie algebra $\overline{\mathfrak g}$ is the content of the following theorem.

\begin{thm}\cite{EFLMK}\label{thm:KLM}
$\mathcal U_*(\mathfrak g)$ is isomorphic, as a Hopf algebra, to $\mathcal U(\overline{\mathfrak g})$. More precisely, the identity map $\operatorname{id}:\overline{\mathfrak g}\rightarrow\mathfrak g$ admits a unique extension to an isomorphism of Hopf algebras $\phi:\mathcal U(\overline{\mathfrak g})\rightarrow \mathcal U_*(\mathfrak g)$.
\end{thm} 

\begin{rmk}\label{rmk:linkSTSR}
In Section \ref{sect:anotherHA} we will show that when the post-Lie algebra structure is defined by a solution of the modified classical Yang--Baxter equation, the isomophism $\phi$ in Theorem \ref{thm:KLM} can be explicitly described in terms of the Hopf algebra structures on $\mathcal U(\bar{\mathfrak g})$ and $\mathcal U_*(\mathfrak g)$.
\end{rmk}

\smallskip

Before further elaborating on the last remark in the context of reference \cite{RSTS} in the next section, we will show a central property of group-like elements in the completed universal enveloping algebra $\mathcal U(\mathfrak g)$ of the post-Lie algebra~${\mathfrak g}$ and, at the same time, we will give a more explicit (combinatorial) expression for the isomorphism $\phi$.

\medskip

In what follows we use $m_{\cdot}: \mathcal U(\overline{\mathfrak g}) \otimes \mathcal U(\overline{\mathfrak g}) \to \mathcal U(\overline{\mathfrak g})$ to denote the product in $\mathcal U(\overline{\mathfrak g})$, i.e., $m_\cdot(A \otimes B)=A . B$ for any $A,B \in \mathcal U(\overline{\mathfrak g})$. The Hopf algebra isomorphism $\phi: \mathcal U(\overline{\mathfrak g}) \to \mathcal U_*(\mathfrak g)$ in Theorem \ref{thm:KLM} can be described as follows. From the proof of Theorem \ref{thm:KLM} it follows that $\phi$ restricts to the identity on $\mathfrak g \hookrightarrow \mathcal U(\mathfrak g)$. Moreover, for $x_1,x_2,x_3 \in \mathfrak g$ we find
$$
	\phi(x_1 .\, x_2) = \phi(x_1) * \phi(x_2) = x_1 * x_2 =x_1x_2 + x_1 \triangleright  x_2,
$$
and 
\allowdisplaybreaks{
\begin{align}
	\phi(x_1 .\, x_2 .\, x_3) &= x_1 * x_2 * x_3 \\	
	&= x_1(x_2 * x_3) + x_1 \triangleright (x_2 * x_3) \label{recursion}\\
	&=x_1x_2x_3 + x_1(x_2 \triangleright x_3) + x_2(x_1 \triangleright x_3) 
			+ (x_1 \triangleright  x_2)x_3 + x_1 \triangleright(x_2 \triangleright x_3). 
\end{align}}
Equality \eqref{recursion} can be generalized to the following simple recursion for words in $\mathcal U(\overline{\mathfrak g})$ with $n$ letters
\begin{equation}
\label{eq:PHIrecursion1}
	\phi(x_1 .\, x_2 .\, \cdots .\, x_n) = x_1\phi(x_2 .\,  \cdots .\,  x_n)  
								+ x_1 \triangleright \phi(x_2 .\,  \cdots .\, x_n) .
\end{equation}
Recall that $x \triangleright \un=0$ for $x \in \mathfrak g$, and $\phi(\un)=\un$. From the fact that the post-Lie product on $\mathfrak g$ defines a linear map $d:\mathfrak g \rightarrow \operatorname{Der}\big(\mathcal U(\mathfrak g)\big),$ we deduce that the number of terms on the righthand side of the recursion \eqref{eq:PHIrecursion1} is given with respect to the length $n=1,2,3,4,5,6$ of the word $x_1 .\, \cdots .\, x_n \in \mathcal U_*(\mathfrak g)$ by 1, 2, 5, 15, 52, 203, respectively. These are the Bell numbers $B_i$, for $i=1,\ldots,6$, and for general $n$, these numbers satisfy the recursion $B_{n+1} = \sum_{i=0}^n {n \choose i} B_i$. Bell numbers count the different ways the set $[n]$ can be partition into disjoint subsets.

From this we deduce the general formula for $x_1  .\,  \cdots  .\, x_n \in \mathcal U(\overline {\mathfrak g})$
\begin{equation}
	\phi(x_1  .\,  \cdots  .\,  x_n) = x_1 * \cdots * x_n = \sum_{\pi \in P_n} X_\pi \in  \mathcal U( {\mathfrak g})    \label{eq:PHIrecursion2},
\end{equation}
where $P_n$ is the lattice of set partitions of the set $[n]=\{1,\dots,n\}$, which has a partial order of refinement ($\pi \leq \kappa$ if $\pi$ is a finer set partition than $\kappa$). Remember that a partition $\pi$ of the (finite) set $[n]$ is a collection of (non-empty) subsets $\pi=\{\pi_1,\dots,\pi_b\}$ of $[n]$, called blocks, which are mutually disjoint, i.e., $\pi_i \cap \pi_j=\emptyset$ for all $i\neq j$, and whose union $\cup_{i=1}^b \pi_i =[n]$. We denote by $|\pi|:=b$ the number of blocks of the partition $\pi$, and $|\pi_i|$ is the number of elements in the block $\pi_i$. Given $p,q \in [n]$ we will write that $p \sim_{\pi} q$ if and only if they belong to same block. The partition $\hat{1}_n = \{\pi_1\}$ consists of a single block, i.e., $|\pi_1|=n$. It is the maximum element in $P_n$. The partition $\hat{0}_n=\{\pi_1,\dots,\pi_n\}$ has $n$ singleton blocks, and is the minimum partition in $P_n$. In the following we denote set-partitions pictorially. For instance, the five elements in $P_3$ are depicted as follows:
$$
	\begin{array}{c}
		\strich\strich\strich \\
		1\ 2\ 3
	\end{array}	
	\qquad\
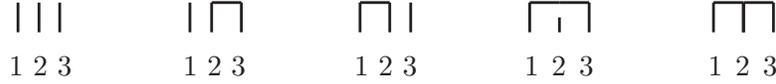
	\begin{array}{c}
	\strich\; \n \\
	\phantom{n} 1\ 2\ 3
	\end{array}
	\qquad\;\;\; 
	\begin{array}{c}
	\n\hspace{0.3cm} \strich \\
	1\ 2\ 3\
	\end{array}
	\qquad\
	\begin{array}{c}
	\nin\\
	1\;\, 2\;\, 3
	\end{array}
	\qquad\;\;
	\begin{array}{c}
	\n\hspace{0.125cm} \n\\
	\phantom{t}1\;\, 2\;\, 3
	\end{array}
$$
The first represents the minimal element in $P_3$, i.e., the singleton partition $\{\{1\},\{2\},\{3\}\}$. The second, third and fourth diagram represent the partitions $\{\{1\},\{2,3\}\}$, $\{\{1,2\},\{3\}\}$, and $\{\{2\},\{1,3\}\}$, respectively. The last one is the maximal element in $P_3$, which consists of a single block $\{\{1,2,3\}\}$. At order 4 we list the examples
$$
	 \strich\, \n\hspace{0.12cm} \n 
	 \qquad\ 
	  \strich \nin 
	 \qquad\ 
	  \n\hspace{0.015cm} \nin
	 \qquad\ 
	  \nin \strich
$$
where the first and second diagram correspond to $\{\{1\},\{2,3,4\}\}$ and $\{\{1\},\{3\},\{2,4\}\}$, respectively. The third and fourth diagram  correspond to $\{\{3\},\{1,2,4\}\}$ and $\{\{2\},\{1,3\},\{4\}\}$, respectively.

Observe that the particular ordering of the blocks in the partitions of the above examples follows from translating the pictorial representation by  ``reading" it from right to left. More precisely, the ordering of the block of any partition $\pi = \{\pi_1, \ldots, \pi_l\}$ associated to the graphical representation, is such that $\max(\pi_i)>j$,  $\forall j \in \pi_m$, $m<i$. Moreover, the elements in each block $\pi_i=\{k_1^i,k_2^i ,\ldots ,k_s^i\}$ are in natural order, i.e., $k_1^i < k_2^i <\cdots  < k_s^i$.  Hence, in the following we assume that the blocks of any partition $\pi$ are in increasing order with respect to the maximal element in each block, and the elements in each block are in natural increasing order, too.

The element $X_\pi$ in \eqref{eq:PHIrecursion2} is defined as follows
\begin{equation}
	X_{\pi} := \prod_{\pi_i \in \pi} x(\pi_i), \label{eq:PHIrecursion2a} 
\end{equation}
where $x(\pi_i):= \ell^{\triangleright }_{x_{k_1^i}} \circ \ell^{\triangleright }_{x_{k_2^i}} \circ\cdots \circ  \ell^{\triangleright }_{x_{k_{l-1}^i}}(x_{k_l^i})$ for the block $\pi_i=\{k_1^i,k_2^i,\ldots ,k_l^i\}$ of the partition $\pi=\{\pi_1, \ldots, \pi_m\}$, and $\ell^{\triangleright}_{a}(b):= a \triangleright b$, for $a,b$ elements in the post-Lie algebra  $\mathfrak g \hookrightarrow \mathcal U(\mathfrak g)$. Recall that $k_l^i \in \pi_i$ is the maximal element in this block. For instance
$$
	X_{\scalebox{0.6}{\strich\strich\strich}}=x_1x_2x_3, 
	\quad
	X_{\scalebox{0.6}{\strich \n}}=x_1(x_2 \triangleright x_3),
	\quad
	X_{\scalebox{0.6}{\nin}}=x_2(x_1 \triangleright x_3),
$$
$$
	X_{\scalebox{0.6}{\n\hspace{0.05cm} \strich }}=(x_1 \triangleright x_2)x_3,
	\quad
	X_{\scalebox{0.6}{\n\hspace{0.00cm} \n}}=x_1 \triangleright(x_2 \triangleright x_3)
$$

\begin{rmk}
Defining $m_i:=\phi(x^{\cdot i})$ and $d_i := \ell^{\triangleright i-1}_x(x):=x \triangleright (\ell^{\triangleright i-2}_x(x))$, $ \ell^{\triangleright 0}=\mathrm{id}$, we find that \eqref{eq:PHIrecursion2} is the $i$-th-order non-commutative Bell polynomial, $m_i = {\mathrm{B}}^{nc}_i(d_1,\ldots,d_i)$. See \cite{ELM14,LMK2} for details. 
\end{rmk}

Next we state a recursion for the compositional inverse $\phi^{-1}(x_1 \cdots x_n)$ of the word $x_1 \cdots x_n \in \mathcal U(\mathfrak g)$. First, it is easy to see that $\phi^{-1}(x_1x_2)=x_1 .\, x_2 - x_1 \triangleright  x_2 \in \mathcal U(\overline{\mathfrak g})$. Indeed, since $\phi$ is linear and reduces to the identity on $\mathfrak g \hookrightarrow \mathcal U(\mathfrak g)$, we have
$$
	\phi(x_1 .\, x_2 - x_1 \triangleright  x_2)= x_1 * x_2 - x_1 \triangleright  x_2 = x_1x_2,
$$
and 
\allowdisplaybreaks{
\begin{align*}
	\phi^{-1}(x_1x_2x_3 ) = x_1 .\, x_2  .\, x_3 
		- \phi^{-1}(x_1(x_2 \triangleright x_3)) 
		- \phi^{-1}(x_2(x_1 \triangleright x_3)) 
		- \phi^{-1}((x_1 \triangleright  x_2)x_3) 
		- x_1 \triangleright(x_2 \triangleright x_3)
\end{align*}}
which is easy to verify. In general, we find a recursive formula for $\phi^{-1}(x_1 \cdots x_n) \in \mathcal U(\overline{\mathfrak g})$
\begin{equation}
\label{eq:PHIrecursion3}
	\phi^{-1}(x_1 \cdots x_n) = x_1  .\, \cdots  .\, x_n - \sum_{\hat{0}_n < \pi \in P_n} \phi^{-1}(X_\pi).
\end{equation}
This is well-defined since in the sum on the righthand side all partitions have less than $n$ blocks.

\medskip

Next we compare group-like elements in the completions of $\mathcal U(\mathfrak g)$ and $\mathcal U_*(\mathfrak g)$, which we denote by $\hat{\mathcal U}(\mathfrak g)$ respectively $\hat{\mathcal U}_*(\mathfrak g)$. 

Recall that if $(H,m,u,\Delta,\epsilon,S)$ is a Hopf algebra and $I=\operatorname{Ker}(\epsilon:H\rightarrow\mathbb K)$ the augmentation ideal, then on $\hat{\mathcal H}:=\varprojlim H/I^n$ can be defined the structure of \emph{complete Hopf algebra}. The elements of $\hat H$ are the Cauchy sequences $\{x_n\}_{n\geq 0}$ with respect to the topology generated by $\{V_n(x)=x+I^n\,\vert\,x\in H\}_{n\geq 0}$. In particular, in $\hat H$ one finds elements of the form $\operatorname{exp}(\xi):=\sum_{\geq 0}\frac{\xi^n}{n!}$, and one can prove that $x\in\hat H$ is \emph{primitive}, i.e., $\hat{\Delta}(x)=x\hat{\otimes} {\bf{1}}+{\bf{1}}\hat{\otimes} x$, if and only if $\operatorname{exp}(x)$ is \emph{group-like}, that is, $\hat{\Delta}(\operatorname{exp}(x))=\operatorname{exp}(x)\hat{\otimes}\operatorname{exp}(x)$ \cite{Quillen}. Note that the set $\mathcal G(\hat H)$ of group-like elements forms a group with respect to the associative product of $\hat H$, and that for every $\xi \in \mathcal G(\hat H)$, $\xi^{-1}=\hat S(\xi)$. Moreover, note that the set of primitive elements $\mathcal P(\hat H)$ forms a Lie algebra whose Lie bracket is defined by anti-symmetrizing the associative product of $\hat H$. The map $\operatorname{exp}:\mathcal P(\hat H)\rightarrow\mathcal G(\hat H)$, $x\mapsto\operatorname{exp}(x)$, is a bijection of sets whose inverse defines the logarithm function. Let $H=\mathcal U(\mathfrak g)$, the universal enveloping algebra of $\mathfrak g$, and consider its completion $\hat{\mathcal U}(\mathfrak g)$. Since $\mathfrak g=\mathcal P(\hat{\mathcal U}(\mathfrak g))$, one deduces the existence of a bijection between $\mathfrak g$ and the group $\mathcal G(\hat{\mathcal U}(\mathfrak g))$, which to every primitive element $x \in \mathfrak g$ associates the corresponding unique group-like element $\operatorname{exp}(x)$. Note that in the process of completing ${\mathcal U}(\mathfrak g)$, the Lie algebra $\mathfrak g$ is completed as well \cite{Quillen}.

Observe that $\phi$ maps the augmentation ideal of $\mathcal U(\overline{\mathfrak g})$ to the augmentation ideal of $\mathcal U(\mathfrak g)$. Therefore, it extends to an isomorphism $\hat\phi : \hat{\mathcal U}(\overline{\mathfrak g}) \to \hat{\mathcal U}(\mathfrak g)$ of complete Hopf algebras.

We are interested in the inverse of the group-like element $\exp(x) \in \mathcal G(\hat{\mathcal U}(\mathfrak g))$ with respect to $\hat{\phi}$. It follows from the inverse of the word $x^n \in \hat{\mathcal U}(\mathfrak g)$, i.e., $\hat{\phi}^{-1}(\exp(x))=\sum_{n \ge 0} \frac{1}{n!} \hat\phi^{-1}(x^n)$.  The central result is the following 

\begin{thm}\label{thm:FinverseChi}
For each $x \in \mathfrak g$, there exists an unique element $\chi(x) \in \mathfrak g$, such that 
\begin{equation}
	\exp(x) = \exp^*(\chi(x)). \label{group-like}
\end{equation}
\end{thm}

\begin{proof}
For $x \in \mathfrak g$ the exponential $\exp(x)$ is a group-like element in $\mathcal G(\hat{\mathcal U}(\mathfrak g))$. The proof of Theorem \ref{thm:FinverseChi} involves calculating the inverse of the group-like element $\exp(x) \in \mathcal G(\hat{\mathcal U}(\mathfrak g))$ with respect to the map $\hat{\phi}$. Indeed, we would like to show that $\hat{\phi}^{-1}(\exp(x)) = \exp^\cdot(\chi(x)) \in \mathcal G(\hat{\mathcal U}(\bar{\mathfrak g}))$, from which identity \eqref{group-like} follows
$$
	\hat{\phi}\circ\hat{\phi}^{-1}(\exp(x)) = \exp(x) = \hat{\phi}\circ \exp^\cdot(\chi(x)) = \exp^*(\chi(x)),
$$
due to $\hat\phi$ being an algebra morphism from $\hat{\mathcal U}(\overline{\mathfrak g})$ to $\hat{\mathcal U}_*(\mathfrak g)$, which reduces to the identity on~${\mathfrak g}$.

First we show that for $x\in \mathfrak g$, the element $\chi(x)$ is defined inductively. For this we consider the expansion $\chi(xt):=xt + \sum_{m>0} \chi_m(x)t^m$ in the dummy parameter $t$. Comparing $\exp^*(\chi(xt))$ order by order with $\exp(xt)$ yields at second order in $t$
$$
	 \chi_2(x) := \frac{1}{2}x_1x_2 - \frac{1}{2}x_1 * x_2=- \frac{1}{2}x \triangleright x \in \mathfrak g. 
$$
At third order we deduce from \eqref{group-like} that 
\allowdisplaybreaks{
\begin{align*}
	\lefteqn{\chi_3(x) := -\frac{1}{3!} \sum_{\hat{0}_3 < \pi \in P_3} X_\pi   - \frac{1}{2}  \chi_2(x) * x - \frac{1}{2}  x * \chi_2(x)} \\
		      &= -\frac{1}{3!} \sum_{\hat{0}_3 < \pi \in P_3} X_\pi  
		      			+ \frac{1}{4}  \big((x \triangleright x) x + (x \triangleright x) \triangleright x\big)
					+ \frac{1}{4}  \big(x (x \triangleright x) + x \triangleright (x \triangleright x)\big)\\
		      &= -\frac{1}{3!}\big( 2x(x \triangleright x) + (x \triangleright  x)x + x \triangleright (x \triangleright x)\big)
		      			+ \frac{1}{4}  \big((x \triangleright x) x + (x \triangleright x) \triangleright x\big)
					+ \frac{1}{4}  \big(x (x \triangleright x) + x \triangleright (x \triangleright x)\big)\\
		     &= \frac{1}{12} [(x \triangleright x), x] 
		     			+ \frac{1}{4} (x \triangleright x) \triangleright x 
		     				+  \frac{1}{12} x \triangleright (x \triangleright x) \in \mathfrak g\\
		     &=  \frac{1}{6} [\chi_1(x), \chi_2(x)] 
		     			- \frac{1}{2} \chi_2(x)\triangleright x 
		     				-  \frac{1}{6} x \triangleright \chi_2(x),		 
\end{align*}}
where we defined $\chi_1(x):=x$. The $n$-th order term is given by
\allowdisplaybreaks{
\begin{align}
\label{eq:nth-order}
	\chi_n(x) &:= -\frac{1}{n!} \sum_{\hat{0}_n < \pi \in P_n} X_\pi  
		- \sum_{k=2}^{n-1} \frac{1}{k!} \sum_{p_1 + \cdots + p_k = n \atop p_i > 0}  \chi_{p_1}(x) *  \chi_{p_2}(x) * \cdots *  \chi_{p_k}(x)\\
		      &= \frac{1}{n!} x^n -\frac{1}{n!} x^{*n} 
		      - \sum_{k=2}^{n-1} \frac{1}{k!} \sum_{p_1 + \cdots + p_k = n \atop p_i > 0}  \chi_{p_1}(x) *  \chi_{p_2}(x) * \cdots *  \chi_{p_k}(x).\end{align}}
From this we derive an inductive description of the terms $\chi_n(x) \in \hat{\mathcal U}_*({\mathfrak g})$ depending on the $\chi_p(x)$ for $1 \le p \le n-1$
\begin{equation}
	\chi_n(x) := \frac{1}{n!} x^n 
		      - \sum_{k=2}^{n} \frac{1}{k!} \sum_{p_1 + \cdots + p_k = n \atop p_i > 0}  \chi_{p_1}(x) *  \chi_{p_2}(x) * \cdots *  \chi_{p_k}(x).
		      \label{chi-map}
\end{equation} 

We have verified directly that the first three terms, $\chi_i(x)$ for $i=1,2,3$, in the expansoin $\chi(xt):=xt + \sum_{m>0} \chi_m(x)t^m$ are in $ \mathfrak g$. Showing that $\chi_n(x) \in \mathfrak g$ for $n>3$ by induction using formula \eqref{chi-map} is surely feasible. However, we follow another strategy. At this stage \eqref{chi-map} implies that $\chi(x) \in \hat{\mathcal U}_*({\mathfrak g})$ exists. Since $x \in \mathfrak g$, we have that $\exp(x)$ is group-like, i.e., $\hat{\Delta} (\exp(x)) = \exp(x) \hat\otimes \exp(x)$. Recall that $\hat{\mathcal U}_*({\mathfrak g})$ is a complete Hopf algebra with the same coproduct $\hat{\Delta}$. Hence
$$
	\hat{\Delta}(\exp^*(\chi(x)))
			= \hat{\Delta} (\exp(x))
			= \exp(x) \hat\otimes \exp(x) 
			= \exp^*(\chi(x)) \hat\otimes \exp^*(\chi(x)).
$$
Using $\hat\phi$ we can write $\hat\phi \hat\otimes \hat\phi \circ \hat{\Delta}_{\overline{\mathfrak g}}(\exp^\cdot(\chi(x))) = \hat\phi \hat\otimes \hat\phi \circ (\exp^\cdot(\chi(x)) \hat\otimes \exp^\cdot(\chi(x))),$ which implies that $\exp^\cdot(\chi(x))$ is a group-like element in $\hat{\mathcal U}(\overline{\mathfrak g})$
$$
	 \hat{\Delta}_{\overline{\mathfrak g}}(\exp^\cdot(\chi(x))) = \exp^\cdot(\chi(x)) \hat\otimes \exp^\cdot(\chi(x)).
$$ 
Since $\hat{\mathcal U}(\overline{\mathfrak g})$ is a complete filtered Hopf algebra, the relation between group-like and primitive elements is one-to-one \cite{Quillen}. This implies that $\chi(x) \in \overline{\mathfrak g} \simeq {\mathfrak g}$, which proves equality \eqref{group-like}. Note that $\chi(x)$ actually is an element of the completion of the Lie algebra $\mathfrak g$. However, the latter is part of $\hat{\mathcal U}({\mathfrak g})$. 
\end{proof}

\begin{cor}\label{cor:diffeqChi}
Let $x \in {\mathfrak g}$. The following differential equation holds for $\chi(xt) \in \mathfrak g[[t]]$
\begin{equation}
\label{proof-key2}
	\dot \chi(xt) =  {\rm dexp}^{*-1}_{-\chi(xt)}\Big( \exp^*\big(-\chi(xt)\big) \triangleright   x\Big).
\end{equation}
The solution $\chi(xt)$ is called post-Lie Magnus expansion. 
\end{cor}

\begin{proof} Recall the general fact for the $\rm{dexp}$-operator \cite{Blanes}
$$
	\exp^*({-\beta(t)}) \ast \frac{d }{dt}\exp^*({\beta(t)}) = \exp^*({-\beta(t)}) \ast {\rm{dexp}}^\ast _{\beta}(\dot{\beta}) *\exp^*({\beta(t)}) 
	={\rm{dexp}}^\ast _{-\beta}(\dot{\beta}),
$$
where 
$$
	{\rm{dexp}}^\ast _{\beta}(x):= \sum_{n \ge 0} \frac{1}{(n+1)!}ad^{(\ast n)}_\beta(x)
	\qquad {\rm{and}} \qquad
	{\rm dexp}^{\ast  -1}_{\beta}(x):=\sum_{n \ge 0} \frac{b_n}{n!} ad^{(\ast n)}_\beta(x).
$$
Here $b_n$ are the Bernoulli numbers and $ad^{(\ast k)}_a(b):=[a,ad^{(\ast k-1)}_a(b)]_\ast$. This together with the differential equation $\frac{d}{dt}\exp^*(\chi(xt)) = \exp(xt)x$ deduced from \eqref{group-like}, implies
\allowdisplaybreaks{ 
\begin{eqnarray}
	 {\rm{dexp}}^{*}_{-\chi(xt)}\big(\dot \chi(xt)\big) 
	  			 &=& \exp^*\big(-\chi(xt)\big)* (\exp(xt)x) \nonumber\\ 
	 			 &=&  \exp^*\big(-\chi(xt)\big)
				 	\Big(\exp^*\big(-\chi(xt)\big) \triangleright   (\exp(xt)x)\Big) \label{step1}\\					 								&=&	 \exp^*\big(-\chi(xt)\big)
					 \bigg(
						\big(\exp^*\big(-\chi(xt)\big) \triangleright   \exp(xt)\big) 
						 \big(\exp^*\big(-\chi(xt)\big) \triangleright   x\big)
					 \bigg) \label{step2}\\ 
				 &=&	 \exp^*\big(-\chi(xt)\big)
					\bigg(
						 \big(\exp^*\big(-\chi(xt)\big) \triangleright   \exp^*\big(\chi(xt)\big)\big) 
				 		\big(\exp^*\big(-\chi(xt)\big) \triangleright   x\big)
					\bigg) \label{step3}\\ 
				  &=&
				  	\bigg( \exp^*\big(-\chi(xt)\big)
				 		\Big(\exp^*\big(-\chi (xt)\big) \triangleright 
						\exp^*\big(\chi(t a)\big)\Big)  
					\bigg)
						 \big(\exp^*\big(-\chi (xt)\big) \triangleright x\big)
									 \label{step4}\\ 
				 &=&
			 		 \Big( \exp^*\big(-\chi(xt)\big) * 
					 \exp^*\big(\chi(xt)\big) \Big)
				 	\big(\exp^*\big(-\chi(xt)\big) \triangleright   x\big) 	\nonumber\\ 
				 &=& \exp^*\big(-\chi(xt)\big) \triangleright x. 			\nonumber
\end{eqnarray}} 
The claim in \eqref{proof-key2} follows after inverting $ {\rm{dexp}}^{*}_{-\chi(xt)}\big(\dot \chi(xt)\big) $. Note that we used successively \eqref{eq:postLieU}, \eqref{eq:pha2} and \eqref{group-like}
\end{proof}

\begin{rmk} Note that any post-Lie algebra with an abelian Lie bracket becomes to a pre-Lie algebra, and the universal enveloping algebra ${\mathcal U}(\mathfrak g)$ reduces to the symmetric algebra ${\mathcal S}(\mathfrak g)$. This is the setting of \cite{OudomGuin}, and identity \eqref{group-like} was described in the pre-Lie algebra context in \cite{ChapPat}. In this case the post-Lie Magnus expansion $\chi(x)$ restricts to the simpler pre-Lie Magnus expansion \cite{EM,Manchon}     
\end{rmk}

In the next section we further explore the universal enveloping algebra corresponding to a post-Lie algebra defined in terms of a classical $r$-matrix, by looking at group-like elements in the completed universal enveloping algebra $\hat{\mathcal U}(\mathfrak g)$.


\section{An isomorphism theorem}
\label{sect:anotherHA}

In this section we will show that, after specializing to the case of post-Lie algebras defined by a solution of the MCYBE, one can get an explicit formula for the isomorphism map of Theorem  \ref{thm:KLM}.
Let  $\mathcal U(\mathfrak g)$ and $\mathcal U(\mathfrak g_R)$ be the universal enveloping algebras of $\mathfrak g$ respectively $\mathfrak g_R$. Since $R_\pm : \mathfrak g_R \to \mathfrak g$ are Lie algebra morphisms, $R_\pm[x,y]_R = [R_\pm x,R_\pm y]$, the universal property permits to extend both maps to unital algebra morphisms from $\mathcal U(\mathfrak g_R)$ to $\mathcal U(\mathfrak g)$. We shall use the same notation for the latter, that is, $R_\pm: \mathcal U(\mathfrak g_R) \rightarrow \mathcal U(\mathfrak g)$. Their images are $\mathcal U(\mathfrak g_\pm)$, i.e., the universal enveloping algebras of the Lie sub-algebras of $\mathfrak g_\pm \subset \mathfrak g$.

\begin{prop}\label{pro:lineariso}
The map $F:\mathcal U(\mathfrak g_R)\rightarrow\mathcal U(\mathfrak g)$ defined by:
\begin{equation}
\label{eq:sigma}
	F=m_{\mathfrak g}\circ (\operatorname{id}\otimes S_{\mathfrak g})\circ (R_+\otimes R_-)\circ\Delta_{\mathfrak g_R},
\end{equation}
is a linear isomorphism. Its restriction to $\mathfrak g_R \hookrightarrow \mathcal U(\mathfrak g_R)$ is the identity map. 
\end{prop}

\begin{proof}
Note that $m_{\mathfrak g}$ and $S_{\mathfrak g}$ denote product respectively antipode in $\mathcal U(\mathfrak g)$, whereas $\Delta_{\mathfrak g_R}$ denotes the coproduct in $\mathcal U(\mathfrak g_R)$. This slightly more cumbersome notation is applied in order to make the presentation more traceable. Given an element $x \in\mathfrak g_R \hookrightarrow \mathcal U(\mathfrak g_R)$, one has that
\allowdisplaybreaks{ 
\begin{align*}
	F(x) 	&=m_{\mathfrak g}\circ (\operatorname{id}\otimes S_{\mathfrak g})\circ (R_+\otimes R_-)\circ\Delta_{\mathfrak g_R}(x)\\
		&= m_{\mathfrak g}\circ (\operatorname{id}\otimes S_{\mathfrak g})\circ (R_+\otimes R_-)(x \otimes {\bf{1}} + {\bf{1}} \otimes x)\\
		&= m_{\mathfrak g}\circ (\operatorname{id}\otimes S_{\mathfrak g})(R_+(x)\otimes {\bf{1}} + {\bf{1}} \otimes R_-(x))\\
		&= m_{\mathfrak g}(R_+(x)\otimes  {\bf{1}} - {\bf{1}} \otimes R_-(x))\\
		&= R_+(x)-R_-(x) = x \in \mathfrak g,
\end{align*}}%
showing that $F$ restricts to the identity map between $\mathfrak g_R$ and $\mathfrak g$. We use the notation from the foregoing section by writing $m_{\mathfrak g_R}(x \otimes y)=x .\, y$. As in Lemma \ref{lem:coprodast} we have
$$
	\Delta_{\mathfrak g_R}(x_1 .\, \cdots .\, x_n) = x_1 .\, \cdots .\, x_n \otimes {\bf{1}} 
										+ {\bf{1}} \otimes x_1.\,  \cdots .\, x_n
	+ \sum_{k=1}^{n-1}\sum_{\sigma\in\Sigma_{k,n-k}} 
	x_{\sigma(1)}.\, \cdots .\, x_{\sigma(k)} \otimes x_{\sigma(k+1)}.\, \cdots .\, x_{\sigma (n)}.
$$
Since $R_\pm$ are homomorphisms of unital associative algebras, one can easily show that for every $x_{1}.\, \cdots .\, x_k\in\mathcal U_k(\mathfrak g_R)$:
 \begin{eqnarray*}
	F (x_1.\,  \cdots .\, x_k)&=&R_+(x_1)\cdots R_+(x_k) + (-1)^k R_-(x_k)\cdots R_-(x_1) + \\
				&&\sum_{l=1}^{k-1}\sum_{\sigma\in\Sigma_{l,k-l}}(-1)^{k-l}R_+(x_{\sigma(1)})\cdots R_+(x_{\sigma(l)})R_-(x_{\sigma(k)})\cdots R_-(x_{\sigma(l+1)}) \in \mathcal U_k(\mathfrak g),
\end{eqnarray*}
which proves that $F$ maps homogeneous elements to homogeneous elements. To verify injectivity of $F$ one can argue as follows. Since $x=R_+(x)-R_-(x)$ for $x \in \mathfrak g$, one can deduce from the previous formula for $x_1.\, \cdots .\, x_k \in \mathcal U_k(\mathfrak g_R)$ that
\[
	F(x_1 .\, \cdots .\, x_k) = x_1\cdots x_k \; \textbf{mod}\,\mathcal U_{k-1}(\mathfrak g),
\]
where $x_1\cdots x_k$ on the righthand side lies in $\mathcal U_k(\mathfrak g)$. For instance
$$
	F(x_1 .\, x_2) = R_+(x_1)R_+(x_2) + R_-(x_2) R_-(x_1) - R_+(x_1)R_-(x_2) - R_+(x_2) R_-(x_1). 
$$
Using $x+R_-(x)=R_+(x)$ implies in $\mathcal U(\mathfrak g)$ that 
\allowdisplaybreaks{
\begin{align*}
	F(x_1 .\, x_2) &= (x_1 + R_-(x_1))(x_2+R_-(x_2)) + R_-(x_2) R_-(x_1) \\
	& \qquad- (x_1 + R_-(x_1))R_-(x_2) - (x_2+R_-(x_2)) R_-(x_1)\\
	&= x_1x_2 + x_1R_-(x_2) + R_-(x_1)x_2 + R_-(x_1)R_-(x_2) \\
	&+ R_-(x_2) R_-(x_1) - x_1R_-(x_2)  - R_-(x_1)R_-(x_2) - x_2R_-(x_1) - R_-(x_2) R_-(x_1)\\
	& = x_1x_2 + [R_-(x_1),x_2], 
\end{align*}}%
where $x_1x_2 \in  \mathcal U_2(\mathfrak g)$ and $[R_-(x_1),x_2] \in \mathcal  U_1(\mathfrak g) \simeq \mathfrak g$. Then, if $F(x_1 .\, \cdots .\, x_k)=0$, one concludes that $x_1\cdots x_k \in \mathcal U_k(\mathfrak g)$ must be equal to zero, that is, at least one among the elements $x_i \in \mathfrak g$ composing the monomial $x_1\cdots x_k$ is equal to zero. This forces the element $x_1 .\, \cdots .\, x_k \in \mathcal U_k(\mathfrak g_R)$ to be equal to zero, which implies injectivity of $F$.

To prove that the map $F$ is surjective one can argue by induction on the length of the homogeneous elements of $\mathcal U(\mathfrak g)$. The first step of the induction is provided by the fact that $F$ restricted to $\mathfrak g_R$ becomes the identity map, and $ \mathfrak g \hookrightarrow \mathcal U_1(\mathfrak g)$. Suppose now that every element in $\mathcal U_{k-1}(\mathfrak g)$ is in the image of $F$ and observe that  $x_{1}\cdots x_{k}\in\mathcal U_k(\mathfrak g)$ can be written as
\allowdisplaybreaks{
\begin{align*}
	\lefteqn{x_1 \cdots x_k =\prod^k_{i=1}\big(R_+(x_i)-R_-(x_i)\big)}\\
		&= \big(R_+(x_1)-R_-(x_1)\big)\big(R_+(x_2)-R_-(x_2)\big)\prod^k_{i=3}\big(R_+(x_i)-R_-(x_i)\big)\\
		&= \big(R_+(x_1)R_+(x_2)-R_-(x_1)R_+(x_2) - R_+(x_1)R_-(x_2) + R_-(x_1)R_-(x_2)\big)\prod^k_{i=3}\big(R_+(x_i)-R_-(x_i)\big)\\
		&= \Big(R_+(x_{1})\cdots R_+(x_{k}) + (-1)^kR_-(x_{k})\cdots R_-(x_{1})\\
		&+ \sum_{l=1}^{k-1}\sum_{\sigma\in\Sigma_{l.k-l}}
		(-1)^{k-l}R_+(x_{\sigma(1)})\cdots R_+(x_{\sigma(l)})
		\cdot R_-(x_{\sigma(k)}) \cdots R_-({x_{\sigma(l+1)}}) \Big) \textbf{mod}\,\mathcal U_{k-1}(\mathfrak g),
\end{align*}}%
which proves the claim, since 
\allowdisplaybreaks{
\begin{align}
	F(x_1 .\, \cdots.\,  x_k)	
	&= R_+(x_{1}) \cdots R_+(x_{k}) + (-1)^k R_-(x_{k}) \cdots R_-(x_{1}) 	\label{eq:1}\\
	&+ \sum_{l=1}^{k-1} \sum_{\sigma\in\Sigma_{l.k-l}}(-1)^{k-l}
	R_+(x_{\sigma(1)})\cdots R_+(x_{\sigma(l)})\cdot R_-(x_{\sigma(k)}) \cdots R_-(x_{\sigma(l+1)}). \label{eq:2}
\end{align}}
\end{proof}

Using the previous computation and the definition of the $*$-product, one can easily see that $F(x_1 .\, x_2)=x_1x_2 + [R_-(x_1),x_2] =  x_1x_2 + x_1 \triangleright x_2$, where $\triangleright$ is defined in \eqref{def:RBpostLie} (and lifted to $\mathcal U(\mathfrak g) $). It implies that $F(x_1 .\, x_2) = x_1 * x_2 \in \mathcal U_*(\mathfrak g) $. Using a simple induction on the lenght of the monomials, this calculation extends to all of $\mathcal U(\mathfrak g_R)$, which is the content of the following 

\begin{cor}\cite{EFLIMK}\label{cor:isoal}
The map $F$ is an isomorphism of unital, filtered algebras, from $\mathcal U(\mathfrak g_R)$ to $\mathcal{U}_{*}(\mathfrak{g})$. In particular, $F(x_1 .\, \cdots .\, x_n) = x_1 * \cdots * x_n$ for all monomials $x_1 .\, \cdots .\, x_n\in\mathcal U(\mathfrak g_R)$.
\end{cor}

Comparing this result with Theorem \ref{thm:KLM} of the previous section, one has 

\begin{prop}\label{prop:idenF}
If the post-Lie algebra $(\mathfrak g,\triangleright)$ is defined in terms of a classical $r$-matrix $R$ via \eqref{def:RBpostLie}, then the isomorphism $\phi$ of Theorem \ref{thm:KLM} assumes the explicit form given in Formula \eqref{eq:sigma}, i.e. $\phi=F$.
\end{prop}

\begin{proof}
First recall that $\mathfrak g_R=\overline{\mathfrak g}$, see Remark \ref{rmk:doubleLiebracket-post-Lie}. Then note that both $\phi$ and $F$ are isomorphisms of filtered, unital associative algebras taking values in $\mathcal U_*(\mathfrak g)$, restricting to the identity map on $\mathfrak g_R$ which is the generating set of $\mathcal U(\mathfrak g_R)$.
\end{proof}

At this point it is worth making the following observation, which will be useful later.

\begin{cor}\label{cor:dec}
Every $A\in\mathcal U(\mathfrak g)$ can be written uniquely as
\begin{equation}
	A= R_+(A'_{(1)})S_{\mathfrak g}(R_-(A'_{(2)})) \label{eq:factinu1}
\end{equation}
for a suitable element $A'\in\mathcal U(\mathfrak g_R)$, where we wrote the coproduct of this element using Sweedler's notation, i.e., $\Delta_{\mathfrak g_R}(A')=A'_{(1)}\otimes A'_{(2)}$.
\end{cor}

\begin{proof}
The proof follows from \eqref{eq:PHIrecursion3}, where $A':=F^{-1}(A) \in \mathcal U(\mathfrak g_R)$. Proposition \ref{pro:lineariso} then implies that for each $A' \in \mathcal U(\mathfrak g_R)$,
\[
	F(A')= R_+(A'_{(1)})S_{\mathfrak g}(R_-(A'_{(2)})).
\]
\end{proof}

Finally, in this more specialized context, we can give the following computational proof of the result contained in Theorem \ref{thm:KLM}.

\begin{thm}\label{cor:iso}
The map $F:\mathcal U(\mathfrak g_R)\rightarrow\mathcal U_{\ast}(\mathfrak g)$ is an isomorphism of Hopf algebras. 
\end{thm}

\begin{proof}
The map $F$ is a linear isomorphism which sends a monomial of length $k$ to (a linear combination of) monomials of the same length. For this reason the compatibility of $F$ with the co-units is verified.  Since $F:\mathcal U(\mathfrak g_R)\rightarrow\mathcal U_*(\mathfrak g)$ is an isomorphism of filtered, unital, associative algebras, the product $\ast$ defined in \eqref{eq:postLieU} can be defined as the push-forward to $\mathcal U(\mathfrak g)$, via $F$, of the associative product of $\mathcal U(\mathfrak g_R)$
\begin{equation}
	A \ast B=F(m_{\mathfrak g_R}(F^{-1}(A)\otimes F^{-1}(B))),\label{eq:push}
\end{equation}
for all monomials $A,B\in\mathcal U(\mathfrak g)$. This implies immediately the compatibility of $F$ with the algebra units. Let us show that $F$ is a morphism of co-algebras, i.e., that 
\begin{equation}
\label{eq:morco}
	\Delta_{\mathfrak g}\circ F=(F\otimes F)\circ\Delta_{\mathfrak g_R}.
\end{equation}
Corollary \ref{cor:isoal} implies that $F (x_1.\, \cdots .\, x_n)=x_1 \ast \cdots \ast x_n$, and the formula in Lemma \ref{lem:coprodast} yields
\allowdisplaybreaks{
\begin{eqnarray*}
	\Delta_{\mathfrak g}\big(F (x_1 .\, \cdots .\, x_n)\big)
	&=&x_1\ast\cdots\ast x_n\otimes \mathbf{1} +\mathbf{1}\otimes x_1\ast\cdots\ast x_n \\
	&+&\sum_{k=1}^{n-1}\sum_{\sigma\in\Sigma_{k,n-k}}x_{\sigma(1)}\ast\cdots
	 			\ast x_{\sigma(k)}\otimes x_{\sigma(k+1)}\ast\cdots\ast x_{\sigma(n)},
\end{eqnarray*}}%
which turns out to be equal to $(F\otimes F)\circ\Delta_{\mathfrak g_R}(x_1 .\, \cdots .\, x_n).$ The only thing that is left to be checked is that $F$ is compatible with the antipodes of the two Hopf algebras, i.e., that $F\circ S_{\mathfrak g_R}=S_\ast\circ F$, where for $x_1 .\, \cdots .\, x_n\in\mathcal U(\mathfrak g_R)$, $S_{\mathfrak g_R}(x_1 .\, \cdots .\, x_n)=(-1)^n x_n .\, \cdots .\, x_1$. First recall that the antipode is an algebra anti-homomorphism, i.e., $S_\ast (A\ast B)=S_\ast(B)\ast S_\ast (A)$, for all $A,B\in\mathcal U_*(\mathfrak g)$. From this and from the property that $S_{\mathfrak g_R}(x)=-x$ for all $x\in\mathfrak g_R$, using a simple induction on the length of the monomials, one obtains 
\[
	S_\ast (x_1\ast\cdots\ast x_n)=(-1)^nx_n\ast\cdots\ast x_1.
\]
From this observation follows now easily that $F\circ S_{\mathfrak g_R}=S_\ast\circ F$.
\end{proof}

We conclude this section with the following interesting observation, see Remark \ref{rmk:linkSTSR}.

\begin{prop}\label{prop:prodRSTS}
For all $A,B\in\mathcal U(\mathfrak g)$, one has that:
\begin{equation}
\label{eq:pSTS2}
	A \ast B = R_+(A'_{(1)})B S_{\mathfrak g}(R_-(A'_{(2)})),
\end{equation}
where $A' \in \mathcal U(\mathfrak g_R)$ is the unique element, such that $A=F(A')$, see Corollary \ref{cor:dec}.
\end{prop}

\begin{proof}
Let $A', B' \in \mathcal U(\mathfrak g_R)$ such that $F(A')=A$ and $F(B')=B$. We use Sweedler's notation for the coproduct $\Delta_{\mathfrak g_R}(A')=A'_{(1)}\otimes A'_{(2)}$, and write $m_{\mathfrak g_R}(A'\otimes B'):=A' .\, B'$ for the product in $\mathcal U(\mathfrak g_R)$.
\allowdisplaybreaks{
\begin{align*}
	A \ast B=F(A' .\, B')
	&=m_{\mathfrak g}\circ (\operatorname{id}\otimes S_{\mathfrak g})\circ (R_+\otimes R_-) 
				\circ\Delta_{\mathfrak g_R}(A' .\, B')\\
	&=m_{\mathfrak g}\circ(\operatorname{id}\otimes S_{\mathfrak g})\circ (R_+\otimes R_-)(A'_{(1)}\otimes A'_{(2)})\cdot (B'_{(1)}\otimes B'_{(2)})\\
	&=m_{\mathfrak g}\circ(\operatorname{id}\otimes S_{\mathfrak g})\circ (R_+\otimes R_-)(A'_{(1)}\cdot B'_{(1)})\otimes (A'_{(2)}\cdot  B'_{(2)})\\
	&=m_{\mathfrak g}\circ(\operatorname{id}\otimes S_{\mathfrak g})\big(R_+(A'_{(1)})R_+(B'_{(1)})\otimes R_-(A'_{(2)})R_-(B'_{(2)})\big)\\
	&\stackrel{(a)}{=}m_{\mathfrak g}\big(R_+(A'_{(1)}) R_+(B'_{(1)})\otimes S_{\mathfrak g}(R_-(B'_{(2)}))S_{\mathfrak g}(R_-(A'_{(2)}))\big)\\
	&=R_+(A'_{(1)}) R_+(B'_{(1)})S_{\mathfrak g}(R_-(B'_{(2)}))S_{\mathfrak g}(R_-(A'_{(2)}))\\
	&= R_+(A'_{(1)}) F(B')S_{\mathfrak g}(R_-(A'_{(2)})\\
	&=R_+(A'_{(1)}) BS_{\mathfrak g}(R_-(A'_{(2)}),
\end{align*}}
which proves the statement. In equality $(a)$ we applied that $S_{\mathfrak g}(\xi\eta)=S_{\mathfrak g}(\eta)S_{\mathfrak g}(\xi)$.
\end{proof}

The map \eqref{eq:sigma} was first defined in \cite{STS3} (see also \cite{RSTS}), where it was used to push-forward to $\mathcal U(\mathfrak g)$ the associative product of $\mathcal U(\mathfrak g_R)$ using formula \eqref{eq:push}. From the equality between the maps $\phi$ and $F$, see Proposition \ref{prop:idenF}, it follows at once that the associative product $m_\ast$ defined in $\mathcal U(\mathfrak g)$ is the product given in \eqref{eq:postLieU}. Our approach provides an easily computable formula for this product, and does not requires the knowledge of the inverse of the map $F$.

{\bf{Another proof of the Theorem \ref{thm:KLM0}}}\quad We give an alternative proof of Theorem \ref{thm:KLM0}, stating that $\mathcal U_{\ast}(\mathfrak g):=(\mathcal U(\mathfrak g),m_\ast,u_{\mathfrak g},\Delta_{\mathfrak g},\epsilon_{\mathfrak g},S_{\ast})$ is a Hopf algebra. Recall that the original proof, which was based on \cite{OudomGuin}, has as a starting point the explicit form of the extension to $\mathcal U(\mathfrak g)$ of the the post-Lie product, see \eqref{eq:postLieU}. In what follows, we will use instead the linear isomorphism $F$ between $\mathcal U(\overline{\mathfrak g})$ and $\mathcal U(\mathfrak g)$, provided in formula \eqref{eq:sigma}, when the post-Lie algebra is defined in terms of a classical $r$-matrix. Starting from this isomorphism, we will define the $\ast$-product on $\mathcal U(\mathfrak g)$ via formula \eqref{eq:push}, and we will then prove that this can be completed to a Hopf algebra structure. First, note that the unit, coproduct and counit are the same as those defining the usual Hopf algebra structure of $\mathcal U(\mathfrak g)$, which, to simplify notation, will be denoted as $u$, $\Delta$ and $\epsilon$, respectively. To prove the theorem we should first check that $(\mathcal U(\mathfrak g),m_\ast,u_{\mathfrak g},\Delta_{\mathfrak g},\epsilon_{\mathfrak g})$ is a bialgebra. To this end, note that from formula \eqref{eq:pSTS2} one deduces easily that $u_{\mathfrak g}$ is the unit of the algebra $(\mathcal U(\mathfrak g),m_\ast)$. Then, it suffices to prove that $\Delta$ and $\epsilon$ are algebra morphisms, i.e., that 
$\epsilon\otimes\epsilon=\epsilon\circ m_\ast$, which is easily checked, and
\begin{equation}
\label{eq:copr}  
	\Delta\circ m_\ast=m_\ast\otimes m_\ast\circ (\operatorname{id}\otimes\tau\otimes\operatorname{id})
	\circ\Delta\otimes\Delta,
\end{equation}
where $\tau$ is the usual flip map. See \cite{Sweedler} for example. Let us show that \eqref{eq:copr} holds. Recall that 
\[
	m_{\ast}(A\otimes B)=F\big(m_{\mathfrak g_R}(F^{-1}(A)\otimes F^{-1}(B))\big),\qquad \forall A,B\in\mathcal U(\mathfrak g).
\]
For every $A\in\mathcal U(\mathfrak g)$, we will write $\Delta(A)=A_{(1)}\otimes A_{(2)}$. Then the righthand side of \eqref{eq:copr}, when applied to $A\otimes B$, becomes:
\allowdisplaybreaks{
\begin{eqnarray*}
\lefteqn{(m_\ast\otimes m_\ast)\circ (\operatorname{id}\otimes\tau\otimes\operatorname{id})\circ (\Delta\otimes\Delta)(A\otimes B)}\\
	&=&(m_\ast\otimes m_\ast)\circ (\operatorname{id}\otimes\tau\otimes\operatorname{id}) 
						\big((A_{(1)}\otimes A_{(2)})\otimes(B_{(1)}\otimes B_{(2)})\big)\\
	&=&(m_\ast\otimes m_\ast) \big((A_{(1)}\otimes B_{(1)})\otimes(A_{(2)}\otimes B_{(2)})\big)\\
	&=&m_{\ast}(A_{(1)}\otimes B_{(1)})\otimes m_{\ast}(A_{(2)}\otimes B_{(2)}).
\end{eqnarray*}}
On the other hand, computing $(\Delta\circ m_{\ast})(A\otimes B)$, and using that $F$ is a comorphism, one gets
\allowdisplaybreaks{
\begin{eqnarray*}
	\lefteqn{\Delta \big(m_\ast (A\otimes B)\big)}\\
		&=& \Delta\big(F(m_{\mathfrak g_R}(F^{-1}(A)\otimes F^{-1} (B))\big)\\
		&=& F\otimes F\Big(\Delta\big(m_{\mathfrak g_R}(F^{-1}(A)\otimes F^{-1}(B))\big)\Big)\\
		&=& (F\otimes F)\circ(m_{\mathfrak g_R}\otimes m_{\mathfrak g_R})\circ 
		(\operatorname{id}\otimes\tau\otimes\operatorname{id})\circ (\Delta\otimes\Delta)\big(F^{-1}(A)\otimes F^{-1} (B)\big)\\
		&=& (F \otimes F)\circ(m_{\mathfrak g_R}\otimes m_{\mathfrak g_R})
		\circ (\operatorname{id}\otimes\tau\otimes\operatorname{id})
			\big(F^{-1}(A_{(1)})\otimes F^{-1}(A_{(2)})\otimes F^{-1}(B_{(1)})\otimes F^{-1}(B_{(2)})\big)\\
		&=& F\big(m_{\mathfrak g_R}(F^{-1}(A_{(1)})\otimes F^{-1}(B_{(1)}))\big)\otimes 
			F\big( m_{\mathfrak g_R}(F^{-1}(A_{(2)})\otimes F^{-1}(B_{(2)}))\big)\\
		&=& m_{\ast}(A_{(1)}\otimes B_{(1)})\otimes m_{\ast}(A_{(2)}\otimes B_{(2)}),
\end{eqnarray*}}%
which gives the proof of the compatibility between $m_\ast$ and $\Delta$, and concludes the proof of the statement. Regarding the proof of the theorem, it suffices now to show that $S_{\ast}$ defined in \eqref{eq:antipodests} is the antipode, i.e., that it satisfies
$m_\ast\circ (\operatorname{id}\otimes S_\ast)\circ\Delta=u\circ\epsilon=m_\ast\circ (S_\ast\circ\operatorname{id})\circ\Delta$. To this end it is enough to recall that $\Delta({\bf{1}})={\bf{1}}\otimes {\bf{1}}$ and $\Delta(x)=x\otimes {\bf{1}}+{\bf{1}}\otimes x$, for all $x\in\mathfrak g$. From these follow that $S_\ast ({\bf{1}})={\bf{1}}$ and, respectively, that $S_\ast(x)=-x$, for all $x\in\mathfrak g$. Using a simple induction on the length of the monomials it follows that $S_\ast$ satisfies \eqref{eq:antipodests}.


\section{Factorization theorems}
\label{sect:factorThm}

Next we consider Theorem \ref{thm:factorizationtheorem} in the context of the universal enveloping algebra of $\mathfrak g$. To this end, one needs first to trade $\mathcal U(\mathfrak g)$ for its completion $\hat{\mathcal U}(\mathfrak g)$. Also, we assume that the classical $r$-matrix in \eqref{eq:impc} satisfies $R \circ R=\operatorname{id}$, which is equivalent to $R_\pm \circ R_\pm = R_\pm$. 

We observe that, since $R_\pm:\mathcal U(\mathfrak g_R) \rightarrow \mathcal U(\mathfrak g)$ are algebra morphisms they map the augmentation ideal of $\mathcal U(\mathfrak g_R)$ to the augmentation ideal of $\mathcal U(\mathfrak g)$  and, for this reason, both these morphisms extend to morphisms $R_\pm:\hat{\mathcal U}(\mathfrak g_R) \rightarrow \hat{\mathcal U}(\mathfrak g)$. In particular, the map $F$ extends to an isomorphism of (complete) Hopf algebras $\hat{F}:\hat{\mathcal U}(\mathfrak g_R)\rightarrow \hat{\mathcal U}_{\ast}(\mathfrak g)$, defined by 
\[
	\hat{F}=\hat{m}_{\mathfrak g}\circ (\operatorname{id} 
	\hat\otimes \hat{S}_{\mathfrak g})\circ(R_+\hat\otimes R_-)\circ\hat\Delta_{\mathfrak g_R},
\]
where, $\hat\Delta_{\mathfrak g_R}$ denotes the coproduct of $\hat{\mathcal U}(\mathfrak g_R)$, and with $\hat{m}_{\mathfrak g}$, $\hat{S}_{\mathfrak g}$ denoting the product respectively the antipode of $\hat{\mathcal U}(\mathfrak g)$. Let $\operatorname{exp}^{\cdot}(x)\in\mathcal G(\hat{\mathcal U}(\mathfrak g_R))$, $\operatorname{exp}^{\ast}(x)\in\mathcal G(\hat{\mathcal U}_{\ast}(\mathfrak g))$ and
$\exp(x) \in \mathcal G(\hat{\mathcal U}(\mathfrak g))$, the respective exponentials. 

 At the level of universal enveloping algebra, the main result of Theorem \ref{thm:factorizationtheorem} can be rephrased.

\begin{thm}\label{thm:factcircled}
Every element $\operatorname{exp}^{\ast}(x)\in \mathcal G(\hat{\mathcal U}_{\ast}(\mathfrak g))$ admits the unique  factorization:
\begin{equation}
\label{eq:factinu2}
	\operatorname{exp}^{\ast}(x)=\exp({x_+})\exp({-x_-}),
\end{equation}
where $x_\pm := R_\pm x$.
\end{thm}

\begin{proof}
Again, to simplify notation we write $m_{\mathfrak g_R}(x\otimes y)=x . y$, for all $x,y\in\mathfrak g_R$, so that for each $x \in \mathfrak g_R$, $x^{\cdot n}:=x .\, \cdots .\, x$. Then observe that, for each $n\geq 0$, one has
\[
	\hat F (x^{\cdot n})=R_+(x)^n+\sum_{l=1}^{n-1}(-1)^{n-l}{n\choose l}R_+(x)^lR_-(x)^{n-l}+(-1)^nR_-(x)^n.
\]
Then, after reordering the terms, one finds $\hat{F} (\operatorname{exp}_{\cdot}(x))=\exp({x_+})\exp({-x_-}).$ On the other hand, since $\hat F:\hat{\mathcal U}(\mathfrak g_R)\rightarrow \hat{\mathcal U}_\ast(\mathfrak g)$ is an algebra morphism, one obtains for each $n\geq 0$, $\hat F(x^{\cdot n})	= \hat F(x)\ast\cdots\ast \hat F (x) = x^{\ast n},$ from which it follows that $\hat{F} (\operatorname{exp}^{\cdot}(x))=\operatorname{exp}^{\ast}(x),$ giving the result. Uniqueness follows from $R_\pm$ being idempotent.
\end{proof}

The observation in Theorem \ref{thm:FinverseChi} implies for group-like elements in $\mathcal G(\hat{\mathcal U}(\mathfrak g))$ and $\mathcal G(\hat{\mathcal U}_*(\mathfrak g))$ that $\exp(x) = \exp^*(\chi(x))$, from which we deduce 

\begin{prop} \label{prop:fact}
Group-like elements $\operatorname{exp}(x) \in \mathcal G(\hat{\mathcal U}(\mathfrak g))$ factorize  uniquely
\begin{equation}
\label{eq:factast}
	\operatorname{exp}(x)=\exp({\chi_+(x)})\exp({-\chi_-(x)}).
\end{equation}
\end{prop}

\begin{proof}
This follows from Theorem \ref{thm:FinverseChi} and Theorem \ref{thm:factcircled} together with $R_-$ being idempotent.
\end{proof}

\begin{rmk} Looking at $\chi(x)$ in the context of $\hat{\mathcal U}(\mathfrak g)$, i.e., with the post-Lie product on $\mathfrak g$ defined in terms of the classical $r$-matrix, $x \triangleright y = [R_-(x),y]$, we find that $\chi_2(x) = -\frac{1}{2} [R_-(x),x]$ and 
$$
	\chi_3(x) = \frac{1}{4} [R_-([R_-(x),x]),x] + \frac{1}{12} ( [[R_-(x),x], x] + [R_-(x),[R_-(x),x]]).   				
$$
This should be compared with Equation (7) in \cite{EGM}, as well as with the results in \cite{EFLIMK}. In fact, comparing with \cite{EGM}, the uniqueness of \eqref{eq:factast} implies that the post-Lie Magnus expansion $\chi: \mathfrak g \to \mathfrak g$ satisfies the BCH-recursion
$$
	\chi(x) =  x + \overline{\operatorname{BCH}}\big(-R_- (\chi(x)),x\big),
$$
where
\[
	\overline{\operatorname{BCH}}(x,y) = \operatorname{BCH}(x,y)  - x - y = \frac{1}{2} [x,y] + \frac{1}{12} \big[x,[x,y]\big] 
		 				+ \frac{1}{12} \big[y,[y,x]\big] - \frac{1}{24} \big[y,[x,[x,y]]\big] + \cdots.
\]

\end{rmk}

\smallskip


\section{Conclusion}
\label{sect:conclusion}

The paper at hand explores in more detail the properties of post-Lie algebras by analyzing the corresponding universal enveloping algebras. A factorization theorem of group-like elements in (a suitable completion of) the universal enveloping algebra corresponding to a post-Lie algebra is derived. It results from the existence of a particular map, called post-Lie Magnus expansion, on the (completion of the post-)Lie algebra. These result are then considered in the context of post-Lie algebra defined in terms of classical $r$-matrices. The link between the theory of post-Lie algebras and results presented in references \cite{RSTS,STS3} are emphasised. More precisely, while in \cite{EFLMK} the existence of an isomorphism between two Hopf algebras naturally associated to every post-Lie algebra was proven by extending results from \cite{OudomGuin}, in the present paper it was shown that the linear isomorphism defined \cite{RSTS,STS3} is indeed a natural example of such an isomorphism between Hopf algebras. This completes the Hopf algebraic picture in \cite{RSTS,STS3}.


\end{document}